\documentclass[11pt]{article}
\usepackage{indentfirst}
\usepackage{mathrsfs}
\usepackage{latexsym,lineno}
\usepackage{algorithm}
\usepackage{enumerate}
\usepackage{algorithmic}
\usepackage{epsfig}
\usepackage{color}
\usepackage{amsmath}\usepackage{fleqn}\usepackage{verbatim}
\usepackage{epsf}
\usepackage{amsthm}\usepackage{graphicx, float}\usepackage{graphicx}
\usepackage{amsfonts}
\usepackage{amssymb}\usepackage{graphpap}
\usepackage{epic}\usepackage{curves}
\usepackage{tikz}
\usepackage{subfigure}
\usepackage{mathrsfs}
\usepackage{pstricks}
\usepackage{hyperref}

\newtheorem{claim}{Claim}

\newtheorem{theorem}{Theorem}[section]

\newtheorem{lemma}{Lemma}[section]

\numberwithin{equation}{section}

\title{\bf Extremal problems for star forests and cliques\thanks{Research was partially supported by the National
Nature Science Foundation of China (grant numbers 12331012)}}
\date{}

\author {  Yongchun Lu$^{1}$,\, Yongchun Lu, \,  Liying Kang\thanks{\em Corresponding author. Email address: lykang@shu.edu.cn (L. Kang), luyongchun@shu.edu.cn (Y. Lu), lyc328az@163.com (Y. Liu)}\\
{\small $^{1}$ Department of Mathematics, Shanghai University,
Shanghai 200444, P.R. China}}

\author {Yongchun Lu$^{1}$\, Yichong Liu$^{1}$ \,    Liying Kang$^{1,2}$\thanks{\em Corresponding author.  Email addresses: lykang@shu.edu.cn (L. Kang),  luyongchun@shu.edu.cn (Y. Lu), lyc328az@163.com (Y. Liu)} \\
{\small $^{1}$Department of Mathematics, Shanghai University,
Shanghai,  China, 200444}\\
{\small$^{2}$Newtouch Center for Mathematics of Shanghai University,
Shanghai,  China, 200444}}

\begin{document}

\maketitle

\begin{abstract}
    Given  a family of graphs $\mathcal{F}$, the  Tur\'{a}n number $ex(n, \mathcal{F})$ denotes the maximum number of edges in any $\mathcal{F}$-free graph on $n$ vertices. Recently, Alon and Frankl  studied of maximum number of edges in an $n$-vertex $\{K_{k+1}, M_{s+1}\}$-free graph, where $K_{k+1}$ is a complete graph on  $k+1$ vertices and $M_{s+1}$ is a matching of $s+1$ edges. They
    determined the exact value of $ex(n, \{K_{k+1},M_{s+1}\})$. In this paper, we extend  the matching $M_{s+1}$
    to  star forest $(s+1)S_l$,  and determine the exact value of $ex(n, \{K_{k+1},(s+1)S_l\})$ for sufficiently    large enough $n$. Furthermore, all the  extremal graphs are obtained.
    
    \medskip

\noindent{\bf AMS (2000) subject classification:}\ \  05C35
\end{abstract}

\section{Introduction}
Let $G=(V,E)$ be a simple graph with vertex set $V=V(G)$ and edge set $E=E(G)$. Set $e(G)=|E(G)|$. If $u,v\in V(G)$ and $uv\in E(G)$, we say $u$ and $v$ are {\sl adjacent}.
Denote by $G[S]$ the graph induced by $S$, and denote by $G\setminus S$ the graph obtained from $G$ by deleting all vertices of $S$ and all edges incident with $S$. For $V_1,V_2\subseteq V(G)$, $E(V_1,V_2)$ denotes the set of edges between $V_1$ and $V_2$ in $G$, and $e(V_1,V_2)=|E(V_1,V_2)|$.

 Given  a family of graphs $\mathcal{F}$, a graph $G$ is called $\mathcal{F}$-free if $G$ does not contain any copy of a graph in $\mathcal{F}$ as a subgraph. The {\sl  Tur\'{a}n number} of a family of graphs  $\mathcal{F}$, denoted as ex$(n, \mathcal{F})$, is the maximum number of edges in an $\mathcal{F}$-free graph $G$ of order $n$. Denote by EX$(n, \mathcal{F})$ the set of $\mathcal{F}$-free graphs with $ex(n, \mathcal{F})$ edges and call a graph in EX$(n, \mathcal{F})$
 {\sl an extremal graph} for $\mathcal{F}$.

Let $K_r$ denote the complete graph on $r$ vertices. A complete subgraph is called a clique. The $r$-vertex independent set is denoted by $\overline{K_r}$. Let $M_s$ denote the matching of $s$ edges, and $S_l$ denote the star on $l+1$ vertices. The Tur\'{a}n graph $T_{k}(n)$ is a complete multipartite graph formed by partitioning the set of $n$ vertices into $k$ subsets with size as equal as possible, and connecting two vertices by an edge if and only if they belong to different subsets. For two vertex disjoint graphs $G$ and $H$, we write $G\cup H$ for the union of $G$ and $H$. The {\sl join} of $G$ and $H$, denoted by $G\vee H$, is the graph obtained by joining every vertex of $G$ to every vertex of $H$.
 We write $k$ disjoint copies of $H$ as $kH$.

Erd\H{o}s and Gallai  \cite{1959-Erdos-Gallai} determined $ex(n, M_{s+1})$ for all positive integers $k$ and $n$.
Lidick\'{y}  et al. \cite{2013-Liu-p62} investigated the Tur\'{a}n number of a star forest, and determined the Tur\'{a}n number of a star forest for sufficiently large $n$. Lan et al. \cite{2019-Lan} gave $ex(n, (s+1)S_l)$ for $k,l\geq 1$ and $n$ is sufficiently large. Li et al. \cite{2022-Li-p112653}  determined $ex(n, (s+1)S_l)$ for all positive integers $k,l$ and $n$.
Recently, Alon and Frankl \cite{2022-Alon-arXiv:2210.15076} determined the exact value of $ex(n,\{K_{k+1},M_{s+1}\})$.

\begin{theorem}(\cite{2022-Alon-arXiv:2210.15076})
   For $n\geq 2s+1$ and $k\geq 2$, $ex(n,\{K_{k+1},M_{s+1}\})=\max\{e(T_k(2s+1)),e(G(n,k))\}$ where $G(n,k)=T_{k-1}(s)\vee \overline{K_{n-s}}$.
\end{theorem}

Katona and Xiao \cite{2023Katona} determined ex$(n, \{K_{k+1}, P_l\})$ if $l>2k+1$.  For a graph $H$ with $\chi(H)>2$,
Gerbner \cite{{2022-Gerbner-arXiv:2211.03272}} determined ex$(n, \{H, M_{k+1}\})$ apart from a constant additive term. Liu and Kang \cite{2024Liu} obtained asymptotical result for ex$(n, \{H, P_l\})$.

In this paper, we  extend $M_{s+1}$ to the star forest $(s+1)S_l$, we determine the exact value of $ex(n,\{K_{k+1},(s+1)S_l\})$. We will
show the following lemma which claims that
there exists a $K_3$-free extremal graph for $S_{l+1}$ on $n$ vertices.
  \begin{lemma} \label{regu}
   Let $l\geq 1$ and $n\geq l^2+2$.

   (1) If $n$ is even,  there exists an  $l$-regular bipartite graph $\mathcal{R}$ on $n$ vertices whose two parts  $V_1, V_2$ have size $\frac{n}{2}$.

   (2) If $n$ is odd, there exists a $K_3$-free $l$-regular graph or a $K_3$-free almost $l$-regular graph $\mathcal{R}$ on $n$ vertices, and there is a partition $V(\mathcal{R})=V_1\cup V_2$ and a vertex $v_0\in V_2$ such that $|V_1|=\lfloor\frac{n}{2}\rfloor, |V_2|=\lceil\frac{n}{2}\rceil$ and $\mathcal{R}[V\setminus v_0]$ is a bipartite graph.
\end{lemma}

The partition $V(\mathcal{R})= V_1\cup V_2$ in Lemma \ref{regu} is called a good partition of $\mathcal{R}$ and $v_0$ is called  an exceptional vertex of $\mathcal{R}$ when $n$ is odd.
We  construct two families of graphs $\mathcal{G}_1(s)$ and $\mathcal{G}_2(s)$ as follows.
 Let $\mathcal{R}$ be an $(n-s)$-vertex $(l-1)$-regular graph or almost $(l-1)$-regular graph
 in Lemma \ref{regu}, $V_1, V_2$ be a good-partition of $\mathcal{R}$ and $v_0$ be the exceptional vertex of $\mathcal{R}$ when $n-s$ is odd.
   A graph in $\mathcal{G}_1(s)$ is obtained from $T_2(s)$ and $\mathcal{R}$ by joining every vertex in the larger part of $T_2(s)$ to every vertex in $V_1$ and joining every vertex in the other  part of $T_2(s)$ to every vertex in $V_2$ $ (V_2\setminus \{v_0\})$ when $n-s$ is even (odd).
  Let $\mathcal{R}_1$ be a bipartite graph on $n-s$ vertices with bipartition $S, T$, where $|S|=\lceil \frac{n-s}{2}\rceil$ and $|T|=\lfloor\frac{n-s}{2}\rfloor$, and  the degree of each vertex in $T$ is $l-1$ and the degree of each vertex in $S$ is not exceeding $l-1$. It is easily seen that graph $\mathcal{R}_1$ always exists from the proof of Lemma \ref{regu}. A graph in  $\mathcal{G}_2(s)$ is obtained from $T_2(s)$ and $\mathcal{R}_1$ by joining  every vertex in the larger part of $T_2(s)$ to every vertex in $S$ and joining every vertex in the other  part of $T_2(s)$ to every vertex in $T$.

Our main results are the following.
\begin{theorem} \label{main_1}
   Let $k\geq 3,  l\geq 2$, $s\geq 0$ be integers and $n\geq ks^2+(s+1)(l+1)^2$. Then $$ex(n,\{K_{k+1},(s+1)S_l\})=e(T_{k-2}(s))+s(n-s)+\lfloor \frac{(l-1)(n-s)}{2}\rfloor.$$ The graphs $T_{k-2}(s)\vee \mathcal{R}$ are  the extremal graphs for $\{K_{k+1},(s+1)S_l\}$, where $\mathcal{R}$ is a $K_3$-free extremal graph for $S_l$ on $n-s$ vertices.
\end{theorem}

\begin{theorem} \label{main_1remain}
Let $k=2$ and $n$ is large enough.

\noindent
(1) If $l<s+1$, then $$\operatorname{ex}(n,\{K_{k+1},(s+1)S_{l}\})=s(n-s),$$ and $K_{s, n-s}$ is the unique extremal graph.

\noindent
(2) If $l\geq s+1$,  $n-s$ is even, then
 $$\operatorname{ex}(n,\{K_{k+1},(s+1)S_l\})= \lceil\frac{s}{2}\rceil\ \lfloor\frac{s}{2}\rfloor+s\lfloor \frac{n-s}{2}\rfloor+\lfloor \frac{(l-1)(n-s)}{2}\rfloor,$$ and $EX(n, \{K_{k+1},(s+1)S_l\})=\mathcal{G}_1(s)$.

 \noindent
  If $l\geq s+1$, $n-s$ is odd, then
  \begin{eqnarray*}
  \operatorname{ex}(n,\{K_{k+1},(s+1)S_l\})&=& \max\{\lceil\frac{s}{2}\rceil\ \lfloor\frac{s}{2}\rfloor+s\lfloor \frac{n-s}{2}\rfloor+\lfloor \frac{(l-1)(n-s)}{2}\rfloor,\\
  & &\lceil\frac{s}{2}\rceil\ \lfloor\frac{s}{2}\rfloor+s\lfloor \frac{n-s}{2}\rfloor+(l-1)\lfloor \frac{n-s}{2}\rfloor\ +\lceil\frac{s}{2}\rceil\},
 \end{eqnarray*}
 any extremal graph for $\{K_{k+1},(s+1)S_l\}$ is in $\mathcal{G}_1(s)$ or in $\mathcal{G}_2(s)$.
\end{theorem}

\section{The Tur\'{a}n number of $\{ K_{k+1},(s+1)S_l\}$ $(k\geq 3)$}
The following lemma is trivial.

\begin{lemma} \label{sl}
   For any $l$ and $n\geq l^2+2$, $ex(n,S_{l+1})= \lfloor \frac{ln}{2} \rfloor$.  The $l$-regular graph or almost $l$-regular graph on $n$ vertices is an extremal graph for $S_{l+1}$.
\end{lemma}


\noindent\textbf{Proof of Lemma \ref{regu}.}
Given a vertex set $V$ of size $n$, we
    divide the vertex set $V$ into $2$ balanced parts $V_1$ and $V_2$, where $|V_1|=\lfloor \frac{n}{2} \rfloor, |V_2|=\lceil \frac{n}{2} \rceil$. Let $\lfloor \frac{n}{2} \rfloor=ml+r, \lceil \frac{n}{2} \rceil=ml+r+c$, where $c=0$ if $n$ is even and $c=1$ if $n$ is odd, $0\leq r<l$. Then we  divide both $V_1$ and $V_2$ into $m+1$ parts, where each of the first $m$ parts contains $l$ vertices while the last part contains $r$ and $r+c$ vertices respectively. Let the parts be $V_{1,1},\cdots,V_{1,m+1}$ and $V_{2,1},\cdots,V_{2,m+1}$, i.e. $|V_{1,i}|=|V_{2,i}|=l,i\in [m]$ and $|V_{1,m+1}|=r, |V_{2,m+1}|=r+c$.
   We label the vertices in $V_{1,m+1}$  be $a_1,\cdots,a_r$, label the vertices in  $V_{2,m+1}$ be $b_1,\cdots, b_r$ if $c=0$ and label the vertices in  $V_{2,m+1}$ be $b_1,\cdots, b_r, v_0$ if $c=1$.
   Then we connect every vertex in $V_{1,i}$ to every vertex in $V_{2,i}$ for $i\in [m]$ and connect every $a_i, i=1,\cdots, r$ to every $b_j, j=1,\cdots, r$.
    Then for any $ i\in [m]$, $V_{1,i}$ and $V_{2,i}$ form an $l$-regular bipartite graph, $\{a_1, a_2, \cdots, a_r\}$ and $\{b_1,\cdots, b_r\}$ form an $r$-regular bipartite graph. Next we apply some changes to construct a $K_3$-free $l$-regular graph or a $K_3$-free almost $l$-regular graph.

    (1) If $n$ is even,
    we select  an edge $xy$ between  $V_{1,i_0}$ and $V_{2,i_0}$ where $x\in V_{1,i_0}, y\in V_{2,i_0}$, $i_0\in [m]$. We delete $xy$ and add two edges $xb_1,ya_1$, then $x$ and $y$ remain degree $l$, the degrees of both $a_1$ and $b_1$ increase by $1$. By applying such an operation $l-r$ times, we change the degrees of both $a_1$ and $b_1$ from $r$ to $l$, and the degrees of other vertices remain the same. Since $n\geq l^2+2$, applying the similar  operation for $(a_i,b_i), 2\leq i\leq r$, we get an $l$-regular bipartite graph $\mathcal{R}$ on the vertex set $V_1\cup V_2$.

   (2)  If $n$ is odd, we  choose an edge $x_1y_1$ between  $V_{1,i_1}$ and $V_{2,i_1}, i_1\in [m]$,  delete it and add two edges $x_1v_0$ and $y_1v_0$. Then $x_1$ and $y_1$ remain degree $l$, the degree of $v_0$ increases by $2$.
    If $l-2>2$, we choose an edge $x_2y_2$ between $V_{1,i_2}$ and $V_{2,i_2}, $ where $i_2\in [m], i_2\not= i_1$.
    Since $n\geq l^2+2, $ continue this operation $\lfloor \frac{l}{2}\rfloor$ times, we change the degree of $v_0$ from $0$ to $l-1$ or $l$, and the degrees of other vertices remain the same. It is easy to check that the resulting graph $\mathcal{R}$ is a $K_3$-free $l$-regular graph or a $K_3$-free almost $l$-regular graph and $\mathcal{R}[V\setminus v_0]$ is a bipartite graph. \qed


\noindent\textbf{Proof of Theorem \ref{main_1}.} We use induction on $s$.
If $s=0$, by Lemmas \ref{sl} and \ref{regu},
$ex(n,\{K_{k+1},S_l\})= ex(n,S_l)=\lfloor \frac{(l-1)n}{2}\rfloor$, and a $K_3$-free $(l-1)$-regular graph or a $K_3$-free almost $(l-1)$-regular graph on $n$ vertices is an extremal graph for $\{K_{k+1},S_l\}$.
 Now suppose the conclusion is true for $0\leq s'\leq s$.
One can verify that $T_{k-2}(s)\vee \mathcal{R}$ is $\{K_{k+1},(s+1)S_l\}$-free, where $\mathcal{R}$ is a $K_3$-free $(l-1)$-regular graph or a $K_3$-free almost $(l-1)$-regular graph  on $n-s$ vertices. Thus
\begin{eqnarray}\label{eq1}
ex(n,\{K_{k+1},(s+1)S_l\})&\geq &e(T_{k-2}(s)\vee \mathcal{R})\nonumber\\
&=& e(T_{k-2}(s))+s(n-s)+\lfloor \frac{(l-1)(n-s)}{2}\rfloor.
\end{eqnarray}
Next we prove the upper bound. Let $\mathcal{G}$ be the set of $n$-vertex $\{K_{k+1},(s+1)S_l\}$-free graphs which maximize the number of edges.
For any graph $G$ in  $\mathcal{G}$, we have the following claim.
\begin{claim}\label{claim_1}
   There exists a vertex set $U\subseteq V(G)$ with size $s$ such that $G[V\setminus U]$ is $S_l$-free.
\end{claim}
\begin{proof}
   By (\ref{eq1})  and $n\geq ks^2+(s+1)(l+1)^2$, we have
   \begin{eqnarray*}
   e(G)&\geq &e(T_{k-2}(s))+s(n-s)+\lfloor \frac{(l-1)(n-s)}{2}\rfloor\\
   &>& e(T_{k-2}(s-1))+(s-1)(n-s+1)+\lfloor \frac{(l-1)(n-s+1)}{2}\rfloor.
    \end{eqnarray*}
    Combining with  the induction hypothesis, we conclude that $G$ must contain a copy of $sS_l$. We denote it by $T_0$.

   We choose an arbitrary copy of $S_l$ in $T_0,$ say $K$. Note that $G\setminus V(K)$ is $sS_l$-free. By the induction hypothesis
    \begin{eqnarray*}
   e(G\setminus V(K))\leq e(T_{k-2}(s-1))+(s-1)(n-s-l)+\lfloor \frac{(l-1)(n-s-l)}{2}\rfloor.
    \end{eqnarray*}
   Let $m_0$ be the number of edges incident to $V(K)$ in $G$. Then
   \begin{eqnarray*}
   m_0&=&e(G)-e(G\setminus V(K))\\
   &\geq &e(T_{k-2}(s))+s(n-s)+\lfloor \frac{(l-1)(n-s)}{2}\rfloor\\
  & & -\left(e(T_{k-2}(s-1))+(s-1)(n-s-l)+\lfloor \frac{(l-1)(n-s-l)}{2}\rfloor\right)\\
   &\geq& n-s,
     \end{eqnarray*}
    which implies that each copy of $S_l$ in $T_0$ must contain a vertex of degree at least $\frac{n-s}{l+1}$. Let $U$ be the set of such a vertex from each $S_l$ in $T_0$, then $|U|=s$. We claim that $G[V\setminus U]$ is $S_l$-free. Otherwise, $G[V\setminus U]$ contains a copy of $S_l$, say $T$. Since  $d(u)\geq \frac{n-s}{l+1}$ for any $u$ in $U$ and  $n\geq ks^2+(s+1)(l+1)^2$, $G\setminus V(T)$ must contain a copy of $sS_l$. Then $G$ contains a copy of $(s+1)S_l$, a contradiction.
\end{proof}
For two non-adjacent vertices $u$ and $v$ in $G$, if we delete the edges connecting $u$ to $N_G(u)$ and add new edges connecting $u$ to $N_G(v)$, we call such an operation as {\sl symmetrization} and we say {\sl symmetrize $u$ to $v$}. Let the resulting graph be $G_{u\rightarrow v}$. Then $V(G_{u\rightarrow v})=V(G)$ and $E(G_{u\rightarrow v})=(E(G)\setminus E(u,N_G(u))\cup E(u,N_G(v))$.

We call two vertices $u$ and $v$ in $V(G)$ are {\sl equivalent} if and only if $N_G(u)=N_G(v)$. Obviously, it is  an equivalent relation.  So the vertices in $U$ can be partitioned into equivalent classes and  every equivalent class is  an independent set. Among $\mathcal{G}$, we choose the graphs that minimize the number of equivalent classes of $U$, and denote the family of such graphs by $\mathcal{G}'$. For any graph $G$ in $\mathcal{G}'$, we have the following claim.
\begin{claim}\label{claim_2}
   Any two non-adjacent vertices of $U$ have the same neighborhood in $G$.
\end{claim}
\begin{proof}
   We first prove that when we do symmetrization for non-adjacent vertices in $U$, the resulting graph is still $\{K_{k+1},(s+1)S_l\}$-free. Suppose $u$ and $v$ are non-adjacent vertices in $U$, and we symmetrize $u$ to $v$, then $G_{u\rightarrow v}[V\setminus U]=G[V\setminus U]$. By Claim \ref{claim_1}, $G_{u\rightarrow v}[V\setminus U]$ is $S_l$-free, then $G_{u\rightarrow v}$ is $(s+1)S_l$-free. Moreover, $G_{u\rightarrow v}$ is $K_{k+1}$-free.
   If $G_{u\rightarrow v}$ contains a copy of  $K_{k+1}$, say $T$, then $T$ must contain $u$. Since $u$ and $v$ are non-adjacent, $(V(T)\setminus \{u\})\cup \{v\}$ will induce a copy of  $K_{k+1}$ in $G$, which contradicts the assumption  that $G$ is $K_{k+1}$-free.

   Next we claim that any non-adjacent vertices in $U$  have the same degree. Otherwise, without loss of generality, suppose $u$ and $v$ are non-adjacent vertices in $U$ but $d(u)<d(v)$. Then we symmetrize $u$ to $v$, $G_{u\rightarrow v}$ is still $\{K_{k+1},(s+1)S_l\}$-free and $e(G_{u\rightarrow v})=e(G)-d(u)+d(v)>e(G)$, which contradicts the assumption that $G\in \mathcal{G}$.

   Suppose there exists non-adjacent vertices $u$ and $v$ in $U$ with $N_G(u)\neq N_G(v)$, then $u$ and $v$ must belong to different equivalent classes $W_1$ and $W_2$. Since $W_1$ is an independent set, for any vertex  $w$ in $W_1$, $d(w)=d(u)=d(v)$. We symmetrize all the vertices in $W_1$ to $v$ one by one, and denote the resulting graph by $G'$, then $e(G)=e(G')$ and $G'$ is still $\{K_{k+1},(s+1)S_l\}$-free. However, in this case $W_1$ and $W_2$ become one equivalent class in $G'$, so the number of equivalent classes decreases, which contradicts the assumption that $G\in \mathcal{G}'$.
\end{proof}
By Claim \ref{claim_2},  we conclude  that $G[U]$ is a complete $m$-partite graph. Let the vertex partition of $G[U]$ be $V_1,\cdots,V_m$, and $|V_1|\geq \cdots\geq |V_m|>0$. Since $G[U]$ is $K_{k+1}$-free, $m\leq k$. For any vertex $v$ in $V\setminus U$, if $v$ is adjacent to one vertex in some $V_i, i\in [m]$, then $v$ is adjacent to all the vertices in $V_i$ by Claim \ref{claim_2}. For each $V_i$, we denote by $N_{G[V\setminus U]}(V_i)$ the common neighborhood of vertices of $V_i$ in $G[V\setminus U]$, and $d_{G[V\setminus U]}(V_i)=|N_{G[V\setminus U]}(V_i)|$.

Next we show that $m=k-2$.
\begin{claim}\label{claim_3}
   $m=k-2$.
\end{claim}
\begin{proof} We first show that $m\geq k-2$. Otherwise, $m\leq k-3$. By Claim \ref{claim_1}, $e(G[V\setminus U])\leq \lfloor \frac{(l-1)(n-s)}{2}\rfloor$. Then we have
   \begin{align}
      e(G)&=e(G[U])+e(U,V\setminus U)+e(G[V\setminus U])\nonumber\\
      &\leq e(T_m(s))+s(n-s)+\lfloor \frac{(l-1)(n-s)}{2}\rfloor\nonumber \\
      &\leq e(T_{k-3}(s))+s(n-s)+\lfloor \frac{(l-1)(n-s)}{2}\rfloor\nonumber\\
      &< e(T_{k-2}(s))+s(n-s)+\lfloor \frac{(l-1)(n-s)}{2}\rfloor\nonumber\\
      &\leq e(G)\nonumber,
   \end{align}
   a contradiction.

   If $m=k-1$, we  claim that there  exists a part $V_{i_0}$ in $U$ such that  $d_{G[V\setminus U]}(V_{i_0})\leq \frac{k-1}{k}(n-s)$. Otherwise, for any $j \in [m-1]$,  $ d_{G[V\setminus U]}(V_j)> \frac{k-1}{k}(n-s)$. This implies that for each $V_j$, there are at most $\frac{1}{k}(n-s)$ vertices in $V\setminus U$ which are not the common neighbor of $V_j$. Then the fact $m=k-1$ implies that there are at least $\frac{1}{k}(n-s)$ vertices in $V\setminus U$ which are  adjacent to all vertices in $U$.
   Let the set of such vertices be $Q$. So
    $|Q|\geq\frac{1}{k}(n-s)$. We claim that $Q$ is an independent set of $G$. Otherwise, an edge from $G[Q]$ and a vertex in each $V_j, j\in [k-1]$ will form a copy of $K_{k+1}$, a contradiction. So $e(G[Q])= 0$. Next we show that $|Q|\leq \frac{k-1}{k}(n-s)$. Otherwise, $V\setminus U$ can be partitioned into two parts, one is $Q$ with size at least $\frac{k-1}{k}(n-s)$ and another is $V\setminus (U\cup Q)$ with size at most $\frac{1}{k}(n-s)$.
    According to  Claim \ref{claim_1} and the fact $Q$ is an independent set of $G$, we get
   \begin{align}
      e(G)&=e(G[U])+e(U,V\setminus U)+e(G[V\setminus U])\nonumber\\
      &\leq e(T_{k-1}(s))+s(n-s)+\frac{1}{k}(l-1)(n-s)\nonumber\\
      &< e(T_{k-2}(s))+s(n-s)+\lfloor \frac{(l-1)(n-s)}{2}\rfloor\nonumber\\
      &\leq e(G)\nonumber
   \end{align}
   for $n\geq ks^2+(s+1)(l+1)^2$, a contradiction. So $\frac{1}{k}(n-s)\leq |Q|\leq \frac{k-1}{k}(n-s)$.
   Thus
   \begin{align}
      e(G)&=e(G[U])+e(U,V\setminus U)+e(G[V\setminus U])\nonumber\\
      &=e(G[U])+e(U,V\setminus (U\cup Q))+e(U,Q)+e(G[V\setminus U])\nonumber\\
      &\leq e(T_{k-1}(s))+(s-|V_{k-1}|)(n-s-|Q|)+s|Q|+ex(n-s,S_l)\nonumber\\
      &\leq e(T_{k-1}(s))+(s-|V_{k-1}|)(n-s)+|V_{k-1}||Q|+\lfloor \frac{(l-1)(n-s)}{2}\rfloor\nonumber\\
      &< e(T_{k-2}(s))+s(n-s)+\lfloor \frac{(l-1)(n-s)}{2}\rfloor\nonumber\\
      &\leq e(G)\nonumber
   \end{align}
    for $n\geq ks^2+(s+1)(l+1)^2$, a contradiction. Therefore there exists a $V_{i_0}$, $1\leq i_0\leq m-1$, such that  $d_{G[V\setminus U]}(V_{i_0})\leq \frac{k-1}{k}(n-s)$. Then
   \begin{align}
      e(G)&=e(G[U])+e(U,V\setminus U)+e(G[V\setminus U])\nonumber\\
      &\leq e(T_{k-1}(s))+(s-|V_{i_0}|)(n-s)+\frac{k-1}{k}|V_{i_0}|(n-s)+\lfloor \frac{(l-1)(n-s)}{2}\rfloor\nonumber\\
      &<e(T_{k-2}(s))+s(n-s)+\lfloor \frac{(l-1)(n-s)}{2}\rfloor\nonumber\\
      &\leq e(G)\nonumber
   \end{align} for $n\geq ks^2+(s+1)(l+1)^2$, a contradiction.

   If $m=k$,  we claim that there  exists a part $V_{i_0}$, $1\leq i_0\leq m$, in $U$ such that $d_{G[V\setminus U]}(V_{i_0})\leq \frac{k}{k+1}(n-s)$. Otherwise,  $ d_{G[V\setminus U]}(V_j)> \frac{k}{k+1}(n-s)$ for any $ j\in [k]$.
   This implies that for each $V_j$, there are at most $\frac{1}{k+1}(n-s)$ vertices in $V\setminus U$ which are not the common neighbor of $V_j$. Then the fact $m=k$ implies that there are at least $\frac{1}{k+1}(n-s)$ vertices in $V\setminus U$ which are  adjacent to all vertices in $U$.
   Let the set of such vertices be $Q$. So
    $|Q|\geq\frac{1}{k+1}(n-s)$.
   However, a vertex in $Q$ and a vertex from each $V_j$, $j\in [k]$
    will form a copy of $K_{k+1}$, a contradiction. Therefore there exists a $V_{i_0}$ in $U$ such that $d_{G[V\setminus U]}(V_{i_0})\leq \frac{k}{k+1}(n-s)$. Then
   \begin{align}
      e(G)&=e(G[U])+e(U,V\setminus U)+e(G[V\setminus U])\nonumber\\[2mm]
      &\leq e(T_{k}(s))+(s-|V_{i_0}|)(n-s)+\frac{k}{k+1}|V_{i_0}|(n-s)+\lfloor \frac{(l-1)(n-s)}{2}\rfloor\nonumber\\[2mm]
      &<e(T_{k-2}(s))+s(n-s)+\lfloor \frac{(l-1)(n-s)}{2}\rfloor\nonumber\\
      &\leq e(G)\nonumber
   \end{align} for $n\geq ks^2+(s+1)(l+1)^2$, a contradiction.
\end{proof}
By Claim \ref{claim_3} and (\ref{eq1}), we conclude that $ex(n,\{K_{k+1},(s+1)S_l\})=e(T_{k-2}(s))+s(n-s)+\lfloor \frac{(l-1)(n-s)}{2}\rfloor$, and
$T_{k-2}(s)\vee \mathcal{R}$ are extremal graphs for $\{K_{k+1},(s+1)S_l\}$,
where $\mathcal{R}$ is a $K_3$-free extremal graph for $S_l$ on $n-s$ vertices. This finishes the proof.\qed

\section{The Tur\'{a}n number of $\{K_3,(s+1)S_l\}$}
 Let $G_1(s)$  and $ G_2(s)$ be graphs in $\mathcal{G}_1(s)$ and $\mathcal{G}_2(s)$ respectively.
Obviously, graphs $ K_{s, n-s}, G_1(s)$ and $ G_2(s)$ all are $\{K_3,(s+1)S_{l}\}$-free,
   \begin{eqnarray*}
   e(G_1(s))&=&\lceil\frac{s}{2}\rceil\ \lfloor\frac{s}{2}\rfloor+s\lfloor \frac{n-s}{2}\rfloor+\lfloor \frac{(l-1)(n-s)}{2}\rfloor,\\[2mm]
   e(G_2(s))&=&\lceil\frac{s}{2}\rceil\ \lfloor\frac{s}{2}\rfloor+(l-1)\lfloor \frac{n-s}{2}\rfloor\
   + \lfloor\frac{s}{2}\rfloor  \lfloor \frac{n-s}{2}\rfloor
   +\lceil\frac{s}{2}\rceil \lceil\frac{n-s}{2}\rceil,
   \end{eqnarray*}
and $\mathcal{G}_1(s)$=$\mathcal{G}_2(s)$
 if $n-s$ is even.

\noindent\textbf{Proof of Theorem \ref{main_1remain}.}
We will show $\operatorname{ex}(n,\{K_3,(s+1)S_l\})= \max\{e(K_{s, n-s}), e(G_1(s)),e(G_2(s))\}$.  Since $K_{s,n-s}, G_1(s)$ and $G_2(s)$ are $\{K_3,(s+1)S_{l}\}$-free,
\begin{align}\label{3.1}
    \operatorname{ex}(n,\{K_3,(s+1)S_l\})\geq  \max\{e(K_{s, n-s}), e(G_1(s)),e(G_2(s))\}.
\end{align}
 Next we use induction on $s$ to prove  $\operatorname{ex}(n,\{K_3,(s+1)S_l\})\leq max\{e(K_{s, n-s}), e(G_1(s)),e(G_2(s))\}$.
 If $s=0$, by  Lemmas \ref{sl} and \ref{regu},
 $$\operatorname{ex}(n,\{K_3,(s+1)S_l\})=e(G_1(0))\leq \max\{e(K_{0,n}),e(G_1(0)),e(G_2(0))\}.$$ Now suppose the conclusion is true for $0\leq s'\leq s$.

 Let $\mathcal{G}$ be the set of $n$-vertex $\{K_3, (s + 1)S_l\}$-free graphs which maximize the number of edges. For any graph $G$ in $\mathcal{G}$, we have the following claim.

\begin{claim}
There exists a vertex set $U\subseteq V(G)$ with size $s$ such that $G[V\setminus U]$ is $S_l$-free.
\begin{proof}
   By (\ref{3.1}) and the induction hypothesis, we have
  \begin{eqnarray*}
      e(G)&\geq& \max\{e(K_{s, n-s}), e(G_1(s)),e(G_2(s))\}\\
      &>& \max\{e(K_{s-1, n-s+1}), e(G_1(s-1)),e(G_2(s-1))\}\\
      &\geq& ex(n,\{K_3,sS_l\}).
  \end{eqnarray*}
   Then $G$ contains  a copy of  $sS_l$. We denote it by $T_1$.
   \par We choose an arbitrary copy of $S_l$ in $T_1$, say $K_1$. Note that $G\setminus V(K_1)$ is $\{K_3, sS_l\}$-free. By the induction  hypothesis, we have
\begin{align}\label{e3.2}
    e(G\setminus V(K_1))\leq \max\{e(K_{s-1, n-s+1}), e(G_1(s-1)),e(G_2(s-1))\}.
\end{align}
\noindent
    Let $m_1$ be the number of edges incident to $V(K_1)$ in $G$. Combining (\ref{3.1}) and (\ref{e3.2}), we get
    \begin{align*}
        m_1&=e(G)-e(G\setminus V(K_1))\\[2mm]
           &\geq \max\{e(K_{s, n-s}), e(G_1(s)),e(G_2(s))\}- \max\{e(K_{s-1, n-s+1}), e(G_1(s-1)),e(G_2(s-1))\}\\[2mm]
           &\geq \min\{e(K_{s, n-s})-e(K_{s-1, n-s+1}), e(G_1(s))-e(G_1(s-1)), e(G_2(s))-e(G_2(s-1))\}\\[2mm]
           &\geq \frac{n}{2}+O(1),
    \end{align*}
    which implies that each copy of $S_l$ in $T_1$ must contain a vertex of degree at least $q=\frac{1}{2(l+1)}n+O(1)$. Let
$U$ be the set of such a vertex from each copy of $S_l$ in $T_1$. Then $|U| = s$ and $G[V\setminus U]$ is $S_l$-free.
\end{proof}
 \end{claim}

Using the approach of Symmetrization as in the proof of Theorem \ref{main_1}, we can prove that
$G[U]$ is a complete $m$-partite graph.  Moreover, since $G[U]$ is $K_3$-free, $m\leq 2$. Let $U_1$ be the larger part in $G[U]$ and $U_2=U\setminus U_1$,  $|U_1|=x$ and $|U_2|=y$. Obviously,  $\lfloor\frac{s}{2}\rfloor\geq y\geq0$. We use $I$ to denote the maximum independent set in $G[V\setminus U]$.

\begin{claim}\label{claim 5}
    If $|I|\leq \lfloor\frac{n-s}{2}\rfloor$, then $e(G)\leq e(G_1(s))$.
    \begin{proof}
        Since $G$ is $K_3$-free,  the neighborhood of each vertex in $U$ is an independent set. Thus the number of edges between $U$ and $V\setminus U$ is at most $s|I|$. Notice that  $G[V\setminus U]$ is $S_l$-free, we have $e(G[V\setminus U])\leq \lfloor \frac{(l-1)(n-s)}{2}\rfloor$ by Lemma \ref{sl}. Hence
\begin{align}\label{eq3.2}
           e(G)&\leq e(G[U])+e(U,V\setminus U)+e(G[V\setminus U])\nonumber\\[2mm]
               &\leq \lceil\frac{s}{2}\rceil\ \lfloor\frac{s}{2}\rfloor+s|I|+\lfloor \frac{(l-1)(n-s)}{2}\rfloor\nonumber\\[2mm]
               &\leq\lceil\frac{s}{2}\rceil\ \lfloor\frac{s}{2}\rfloor+s\lfloor \frac{n-s}{2}\rfloor+\lfloor \frac{(l-1)(n-s)}{2}\rfloor\nonumber\\[2mm]
               &= e(G_1(s)).
        \end{align}
    \end{proof}
\end{claim}

\begin{claim}\label{eu}
    If $|I|>\lfloor\frac{n-s}{2}\rfloor$, then
    \begin{align*}
    e(G)\leq s|I|+(l-1)(n-s-|I|)+y(n-2|I|-y).
    \end{align*}
\begin{proof}
Since $I$ is the maximum independent set in $G[V\setminus U]$ and $G[V\setminus U]$ is $S_l$-free,
\begin{align}\label{3}
e(G[V\setminus U])\leq (l-1)(n-s-|I|)
\end{align}
Let $I_u$ be the neighborhood of vertex $u\in U$ in $G[V\setminus U]$. Then $I_{u_1}\bigcap I_{u_2}=\emptyset$ for any vertex $u_1$ in $U_1$ and any vertex $u_2$ in $U_2$. Otherwise, we can find a $K_3$ copy, a contradiction. This implies that
 $|I_{u_1}|+|I_{u_2}|\leq n-s$. Let $|I_1|=\max\limits_{u\in U_1}I_u$ and $|I_2|=\max\limits_{u\in U_2}I_u$. Then $|I_1|+|I_2|\leq n-s$. So
        \begin{align}\label{3.3}
         e(U,V\setminus U)&=\sum\limits_{u\in U_1}d_{G[V\setminus U]}(u)+\sum\limits_{v\in U_2}d_{G[V\setminus U]}(v)\nonumber\\[2mm]
         &=\sum\limits_{u\in U_1}|I_u|+\sum\limits_{v\in U_2}|I_v|\nonumber\\[2mm]
         &\leq x|I_1|+y|I_2|\nonumber\\[2mm]
         &\leq x|I_1|+y(n-s-|I_1|)\nonumber\\[2mm]
         &\leq(x-y)|I|+y(n-s).
       \end{align}
    By (\ref{3}) and (\ref{3.3}), we get
        \begin{align*}
            e(G)&\leq e(G[U])+e(U,V\setminus U)+e(G[V\setminus U])\\[2mm]
            &\leq xy+(x-y)|I|+y(n-s)+(l-1)(n-s-|I|)\\[2mm]
            &= s|I|+(l-1)(n-s-|I|)+y(n-2|I|-y).
        \end{align*}
\end{proof}
\end{claim}

\begin{claim}\label{en}
    If $|I|>\lfloor\frac{n-s}{2}\rfloor$ and $n$ is large enough, then the following statements holds.
    \begin{description}
\item[(1)] If $l\geq s+1$ and $n-s$ is even,
then
\begin{align*}
\operatorname{ex}(n,\{K_3,(s+1)S_{l}\})
\leq e(G_1(s)).
\end{align*}
\noindent
If $l\geq s+1$ and $n-s$ is odd,
then $$\operatorname{ex}(n,\{K_3,(s+1)S_{l}\})
\leq e(G_2(s)).$$
\item[(2)] If $l<s+1$,
then $$\operatorname{ex}(n,\{K_3,(s+1)S_{l}\})\leq e(K_{s,n-s}).$$
\end{description}
\begin{proof}
  Let $t=n-2|I|$. Applying  Claim \ref{eu}, we obtain
 \begin{align}\label{44}
e(G)&\leq s|I|+(l-1)(n-s-|I|)+y(t-y).
\end{align}
  Invoking the facts $|I|>\lfloor\frac{n-s}{2}\rfloor$ and $n\geq s+|I|$, we get  $-n+2s\leq t\leq s-1$.

 \noindent
 (1) For the case $l\geq s+1$, if
  $t<0$, we obtain from (\ref{44}) that
 \begin{align*}
e(G)& \leq s|I|+(l-1)(n-s-|I|)\\
&\leq (l-1)(n-s)+(s-l+1)(\lfloor\frac{n-s}{2}\rfloor+1)\\
&< e(G_2(s)).
 \end{align*}

 If $0\leq t\leq s-1$ and $n-s$ is even, by (\ref{44}), we get
 \begin{align*}
e(G)&\leq (l-1)(n-s)+(s-l+1)|I|+\lceil\frac{t}{2}\rceil\ \lfloor\frac{t}{2}\rfloor\\[2mm]
    &< \lceil\frac{s}{2}\rceil\ \lfloor\frac{s}{2}\rfloor+(l-1)(n-s)+(s-l+1)(\frac{n-s}{2}+1)\\[2mm]
    &= \lceil\frac{s}{2}\rceil\ \lfloor\frac{s}{2}\rfloor+s\frac{n-s}{2}+(l-1)\frac{n-s}{2}+(s-l+1)\\[2mm]
    &\leq \lceil\frac{s}{2}\rceil\ \lfloor\frac{s}{2}\rfloor+s \frac{n-s}{2}+\frac{(l-1)(n-s)}{2}\\[2mm]
    &=e(G_1(s)).
\end{align*}
   If $0\leq t\leq s-1$ and  $n-s$ is odd, combining with (\ref{44}), we deduce
\begin{align}\label{eq3.6}
    e(G)&\leq (l-1)(n-s)+(s-l+1)(\lfloor\frac{n-s}{2}\rfloor+1)+ \lceil\frac{s-1}{2}\rceil\ \lfloor\frac{s-1}{2}\rfloor\nonumber\\[2mm]
    &=\lceil\frac{s}{2}\rceil\ \lfloor\frac{s}{2}\rfloor+s-\lfloor\frac{s}{2}\rfloor+s\lfloor \frac{n-s}{2}\rfloor+(l-1)(n-s-\lfloor \frac{n-s}{2}\rfloor-1)\nonumber\\[2mm]
    &=\lceil\frac{s}{2}\rceil\ \lfloor\frac{s}{2}\rfloor+s\lfloor \frac{n-s}{2}\rfloor+(l-1)\lfloor \frac{n-s}{2}\rfloor+\lceil\frac{s}{2}\rceil\nonumber\\[2mm]
    &=e(G_2(s)).
\end{align}
(2) For the case $l<s+1$, if $t< 0$, it follows from  (\ref{44}) that
\begin{align}\label{3.2}
    e(G)&\leq (l-1)(n-s)+(s-l+1)|I|\notag\\
    &\leq s(n-s).
\end{align}
If $0\leq t\leq s-1$, notice that $|I|=\frac{n-t}{2}$, we obtain
\begin{align*}
    e(G)&\leq (l-1)(n-s)+(s-l+1)\frac{n-t}{2}+\lceil\frac{s}{2}\rceil\ \lfloor\frac{s}{2}\rfloor\\[2mm]
    &\leq \frac{s+l-1}{2}n+\lceil\frac{s}{2}\rceil\ \lfloor\frac{s}{2}\rfloor\\
    &< s(n-s)
\end{align*}
for sufficiently large $n$.
\end{proof}
\end{claim}

 Using Claims \ref{claim 5} and  \ref{en}, we deduce
 \begin{align*}
 \operatorname{ex}(n,\{K_3,(s+1)S_{l}\})\leq \max\{e(K_{s, n-s}), e(G_1(s)),e(G_2(s))\}.
 \end{align*}
 Combining with (\ref{3.1}), we get
\begin{align*}
 \operatorname{ex}(n,\{K_3,(s+1)S_{l}\}) = \max\{e(K_{s, n-s}), e(G_1(s)),e(G_2(s))\}.
 \end{align*}
 To be more specific,

(1) If $l< s+1$ and $n$ is large enough, then
\begin{align*}
\operatorname{ex}(n,\{K_3,(s+1)S_{l}\})&= \max\{e(K_{s, n-s}), e(G_1(s)),e(G_2(s))\}\\
&=s(n-s).
\end{align*}
Furthermore, according to  the proofs of Claims  \ref{claim 5} and   \ref{en},  we deduce that $e(G)=s(n-s)$ if and only if the equality in (\ref{3.2}) holds, which implies that $|I|=n-s$ and $y=0$. Then $G\cong K_{s, n-s}$, hence  $K_{s, n-s}$ is the unique extremal graph.

(2) If $l\geq s+1$ and $n$ is large enough, then
\begin{align*}
\operatorname{ex}(n,\{K_3,(s+1)S_{l}\})&= \max\{e(K_{s, n-s}), e(G_1(s)),e(G_2(s))\}\\
&=\max\{e(G_1(s)),e(G_2(s))\}.
\end{align*}

When $n-s$ is even, then \begin{align*}
\operatorname{ex}(n,\{K_3,(s+1)S_{l}\})&=\max\{e(G_1(s)), e(G_2(s))\}\\
&=e(G_1(s))=e(G_2(s)).
\end{align*}
Furthermore, we can see from the proofs
of Claims \ref{claim 5} and  \ref{en} that  $e(G)=e(G_1(s))$ if and only if the equality in (\ref{eq3.2}) holds, which implies that $G$ is a graph in $\mathcal{G}_1(s)$.

When $n-s$ is odd, then \begin{align*}
\operatorname{ex}(n,\{K_3,(s+1)S_{l}\})&=\max\{e(G_1(s)), e(G_2(s))\}.
\end{align*}
Furthermore, it follows from   the proofs of Claims \ref{claim 5} and  \ref{en}  that  $e(G) = \max\{e(G_1(s)), e(G_2(s))\}$ if and only if the equality in (\ref{eq3.2}) or the equality in (\ref{eq3.6}) holds,  which implies that $G$ is a graph in $\mathcal{G}_1(s)$ or a graph in $\mathcal{G}_2(s)$.
The proof is finished. \qed

\section*{Declaration of interests}
The authors declare that there is no conflict of interest.

   \end{document}